\documentclass[a4paper, 11pt,reqno]{amsart}

\usepackage{amsmath}
\usepackage{amssymb}
\usepackage{hyperref} 
\usepackage{amsfonts}
\usepackage{mathtools}
\usepackage{theoremref}
\usepackage{amsthm}
\usepackage[dvipsnames]{xcolor}
\usepackage[utf8]{inputenc}
\usepackage{array}
\usepackage{enumitem}
\usepackage[english]{babel}

\usepackage{tikz}
\usetikzlibrary{shapes.geometric}

\calclayout

\setlist{topsep=0mm,partopsep=0mm,itemsep=1mm}

\theoremstyle{plain}
\newtheorem{lemma}{Lemma}[section]
\newtheorem{thm}[lemma]{Theorem}

\newtheorem{prop}[lemma]{Proposition}
\newtheorem*{mainthm}{Main Theorem}
\newtheorem*{thm*}{Theorem}

\theoremstyle{definition}
\newtheorem{cons}[lemma]{Construction}
\newtheorem{que}[lemma]{Question}

\newtheorem{prob}[lemma]{Problem}

\theoremstyle{remark}
\newtheorem{rem}[lemma]{Remark}

\newcommand{\N}{\mathbb{N}}

\newcommand{\Q}{\mathbb{Q}}

\newcommand{\fc}{\mathfrak{c}} %% Had to place this here to avoid making lots of changes, due to the cedilla in the above thanks - AC

\newcommand{\NRrev}[1]{{\color{blue} #1}}

\newcommand{\sdpt}{\textup{\textsf{SDP}}-type}

\DeclareMathOperator{\Div}{\textup{\textsf{Div}}}

\newenvironment{thmenumerate}{\begin{enumerate}[label=\textup{(\roman*)},leftmargin=10mm]}{\end{enumerate}}
\newenvironment{nitemize}{\begin{itemize}[label=\textbullet, leftmargin=5mm]}{\end{itemize}}

\begin{document}

\title[Subdirect powers of commutative semigroups]{On the number of countable subdirect powers of finite commutative semigroups}
\author[A. Clayton, N. Ru\v{s}kuc]{Ashley Clayton \& Nik Ru\v{s}kuc}
\address{School of Mathematics and Statistics, University of St Andrews, St Andrews, Scotland, UK}
\email{$\{$ac323,nr1$\}$@st-andrews.ac.uk}
%t

\begin{abstract}
In 1981/82, Hickin \& Plotkin and McKenzie both proved that a finite group has only countably many non-isomorphic subdirect powers if and only if it is abelian. In this paper, we prove that a finite commutative semigroup has only countably many non-isomorphic countable subdirect powers if and only if it is either a finite abelian group or a null semigroup.
\end{abstract}

\thanks{The first author acknowledges support by FCT - Funda\c{c}\~{a}o para a Ci\^{e}ncia e a Tecnologia, under the project CEMAT- Ci\^{e}ncias (UIDP/04621/2020).
The second author acknowledges support from EPSRC EP/V003224/1.}
\maketitle

\section{Introduction}
\label{sec:intro}

Let $S$ be a semigroup. For a set $X$, the \emph{direct power} of $S$ by $X$ is the set 
$S^X$ of all functions $X\rightarrow S$ with multiplication $(fg)(x)=f(x)g(x)$.
When $X=\{1,\dots,n\}$ we can identify $S^X$ with the set of all $n$- tuples under component-wise multiplication.
Similarly, when $X=\N=\{1,2,\dots\}$, we identify $S^X$ with the set of all infinite sequences $(s_1,s_2,\dots)$ under the component-wise multiplication.

A subsemigroup $T$ of a direct power $S^X$ is said to be a \emph{subdirect power} if $T$ projects \emph{onto} each of the $S$-factors, i.e. if $\{ f(x)\::\: f\in T\}=S$ for all $x\in X$.

In 1981/82 Hickin and Plotkin \cite{hickin1981} and McKenzie \cite{mckenzie1982} proved results, a special case of which can be formulated as follows:

\begin{thm}[Hickin and Plotkin; McKenzie]
\label{thm:hpm}
A finite group has only countably many non-isomorphic countable subdirect powers if and only if it is abelian, else it has continuum many such powers.\qed
\end{thm}

Throughout, by a set being \emph{countable} we mean that it is finite, or that its cardinality is equal to $|\mathbb{N}|=\aleph_0$.
By a countable subdirect power $T\leq S^X$ we mean that $T$ is countable, rather than $X$ being countable. The \emph{continuum} $\fc$ is the cardinality of $\{0,1\}^\N$.

A natural question then arises as to what the situation is for semigroups.
In particular, is it true that finite commutative semigroups have only countably many non-isomorphic subdirect powers?
This turns out to be very far from being the case, and our main result gives a complete classification.
In its statement, a \emph{null semigroup} $S$ is one where there is an element $0\in S$ such that $st=0$ for all $s,t\in S$.

\begin{mainthm}
A finite commutative semigroup $S$ has only countably many non-isomorphic subdirect powers if and only if it is either an abelian group or a null semigroup; otherwise $S$ has continuum many countable subdirect powers. 
\end{mainthm}
 
The rest of the paper is devoted to proving the Main Theorem.
To simplify the terminology, 
we will use the phrase \emph{subdirect power type}, or \sdpt\ for short, to mean the number of non-isomorphic countable subdirect powers of $S$. The considerations split into two parts, depending on the number of idempotents in $S$.
In the case where this number is greater than one, we proceed by showing that $S$ has \sdpt\ $\fc$ when $S$ is: (i) the two-element semilattice (Proposition \ref{prop:U1}); (ii) an arbitrary non-trivial semilattice (Proposition \ref{prop:slatt}); (iii) an arbitrary commutative semigroup with at least two idempotents (Proposition \ref{prop:comm2ids}). The step from (ii) to (iii) is achieved by reference to the decomposition of a semigroup into its \emph{archimedean components}. For the case of a single idempotent, it is an easy observation that null semigroups are of countable \sdpt\ (Proposition \ref{prop:null}).
The remaining considerations are to show that the following are of \sdpt\ $\fc$: (i) $S$ is nilpotent of class $>2$ (Proposition \ref{prop:nilp}); (ii) $S$ is an ideal extension of  
an abelian group by a null semigroup (Proposition \ref{prop:ext}).
%The Main Theorem is a direct consequence of ???.

\section{Preliminaries and outline of proof}
\label{sec:prelim}

In this section we introduce the basic concepts and facts concerning commutative semigroups, and explain how the subsequent results combine to prove the Main Theorem.

A commutative semigroup $S$ is said to be a \emph{semilattice} if all its elements are \emph{idempotents}, i.e.\ $s^2=s$ for all $s\in S$. Alternatively, by letting $s\leq t\Leftrightarrow st=s$, we can view a semilattice as a partially ordered set in which any pair of elements has a greatest lower bound.
In particular, any linearly ordered set can be viewed as a semilattice with the operation $st=\min(s,t)$.
These two definitions are equivalent, and the notion of isomorphism is the same in both interpretations.

A semigroup $S$ with zero is said to be \emph{nilpotent} of \emph{class} $k$ if $s_1s_2\dots s_k=0$ for all $s_1,\dots,s_k\in S$ but $s_1'\dots s_{k-1}'\neq 0$ for some $s_1',\dots,s_{k-1}'\in S$. The non-trivial null semigroups are precisely nilpotent semigroups of class $ 2$.

An \emph{ideal} in a semigroup $S$ is a non-empty set $I$ such that $SIS\subseteq I$. 
The \emph{Rees congruence} associated with an ideal $I$ is
\[
\rho_I=\{ (s,t)\in S\times S\::\: s=t\ \text{ or }\ s,t\in I\}.
\]
The \emph{Rees quotient} $S/I$ is defined as $S/\rho_I$.
We say that a semigroup $S$ is an \emph{ideal extension} of a semigroup $U$ by a semigroup $V$ if 
$S$ has an ideal $I$ such that $I\cong U$ and $S/I \cong V$.

Let $S$ be a commutative semigroup. The relation $\eta$ defined by
\[
s\eta t\ \Leftrightarrow\ s^k=tu,\ t^l=sv\ \text{ for some } k,l\in\N,\ u,v\in S^1
\]
is a congruence. The equivalence classes of $\eta$ are called \emph{archimedean components} of $S$.
If $e$ and $f$ are distinct idempotents then $(e,f)\not\in \eta$. Thus, every archimedean component contains at most one idempotent.
When $S$ is finite then in fact every archimedean component contains precisely one idempotent, which is the common idempotent power of all the elements in the component. Furthermore, in this case, each archimedean component is an ideal extension of a group by a nilpotent semigroup. The quotient $S/\eta$ is a semilattice. In fact $\eta$ is the smallest semilattice congruence on $S$. For a more systematic introduction we refer the reader to \cite[Section IV.2]{grillet}.

Direct and subdirect powers of a semigroup $S$ were introduced in the previous section.
For $s\in S$ and $X\subseteq S$ we let $\overline{s}:=(s,s,\dots)\in S^\N$, and 
$\Delta_X:=\{ \overline{x}\::\: x\in X\}$. Note that $\Delta_S$ is a subdirect power of $S$ isomorphic to $S$; we call $\Delta_S$ the \emph{diagonal copy} of $S$ in $S^\N$.
If $S$ is a finite semigroup then every countable subdirect power of $S$ is already present in $S^\N$. 
Thus in our effort to demarcate finite commutative semigroups of countable \sdpt\ from those of uncountable \sdpt, it is sufficient to consider the subdirect powers inside $S^\N$.

The outline of the proof of the Main Theorem is as follows. 
Suppose $S$ is a finite commutative semigroup.
If $S$ has more than one idempotent, then  \sdpt\ of $S$ is $\fc$ by Proposition \ref{prop:comm2ids}.
So suppose now that $S$ has only a single idempotent. Then $S$ is a single archimedean component of itself, and is an ideal extension of an (abelian) group $G$ by a nilpotent semigroup $V$.
Suppose first that $G$ is trivial, in which case $S\cong S/G\cong V$ is nilpotent.
If the class $k$ of $S$ is $\leq 2$, then $S$ is a null semigroup, and has a countable \sdpt\ by Proposition \ref{prop:null};
otherwise, when $k\geq 3$, the \sdpt\ of $S$ is $\fc$ by Proposition \ref{prop:nilp}.
Now suppose that $G$ is non-trivial. If $V$ is trivial, then $S\cong G$, and it has a countable \sdpt\ by Theorem \ref{thm:hpm}.
Finally, if both $V$ and $G$ are non-trivial, $S$ has \sdpt\ $\fc$ by Proposition \ref{prop:ext},
and this completes the proof of the Main Theorem.

\section{Semilattices}
\label{sec:slatt}

The purpose of this section is to prove that every non-trivial semilattice $S$ has \sdpt\ $\fc$.
This is done by exhibiting uncountably many non-isomorphic subdirect powers inside $S^\N$.
In fact, since this result is a stepping stone towards proving that every commutative semigroup with at least two idempotents has \sdpt\ $\fc$, we prove more, i.e.\ that there are  uncountably many such subdirect powers consisting of certain special kinds of elements we call \emph{recurring}.
These are sequences $\overline{\sigma}\in S^\N$ that can be obtained from finite sequences $\sigma=(s_1,\dots,s_k)$ ($k\in\N$) by concatenating infinitely many copies together, i.e.\ $\overline{\sigma}:=(s_1,\dots,s_k,s_1,\dots,s_k,s_1,\dots)$.

We begin with the smallest non-trivial semilattice $L_2=\{0,1\}$ under the standard multiplication.
The direct power $L_2^\N$ is an uncountable semilattice, in which the ordering is component-wise, i.e.
\[
(s_1,s_2,\dots)\leq (t_1,t_2,\dots) \ \Leftrightarrow\ s_i\leq t_i \text{ for all } i=1,2,\dots.
\]

Let us recall some terminology for linearly ordered sets. The largest and smallest elements in such a set $(A,\leq)$ are referred to as \emph{end-points}; they may or may not exist.
The set $A$ is said to be dense if for any $a,b\in A$ with $a<b$, there exists $c\in A$ such that $a<c<b$.
The set of rational numbers $\Q$ with standard linear ordering is a countable, dense linearly ordered set without end-points. 
In fact it is a unique linearly ordered set with these properties by a well-known theorem due to Cantor; see
\cite[Theorem 2.8]{rosenstein}.
Furthermore, this linearly ordered set contains all countable linearly ordered sets inside it, and there are uncountably many of them up to isomorphism; see \cite[Theorem 2.5, Corollary 2.6]{rosenstein}. 

So, to prove that $L_2$ 
has \sdpt\ $\fc$ we will find a copy of $\Q$ inside $L_2^\N$. This will imply that $L_2^\N$ contains uncountably many non-isomorphic \emph{subsemigroups}, and we will turn those subsemigroups into \emph{subdirect products} by adjoining $\Delta_{L_2}$ to each of them. 
Furthermore, we will be able to accomplish all this by using recurring elements only. 
Note that the set of all recurring elements is countable.
While this fact is not so important here, as we are embedding a fixed countable semilattice, namely $\Q$, the use of recurring elements at this point in the argument will become crucial on Section \ref{sec:sevids}, where it will genuinely ensure the countability of the subdirect products constructed there.

%Though each subdirect product consisting of recurring elements implies their countability, it is not strictly necessary to use recurring elements for that purpose here. Their use at this stage is actually more explicitly justified in \ref{sec:sevids}, where they will genuinely ensure the countability of subdirect powers that we construct from those of $L_{2}$ identified in the following proposition. }
%\ACcomm{I agreed with your previous comment about use of recurring elements here, so I've tried to express your thoughts in some way above. Let me know if this isn't what you meant.}

\begin{prop}
\label{prop:U1}
Let $L_2=\{0,1\}$ be the two-element semilattice, and let $\Q$ be the linearly ordered set of rational numbers.
\begin{thmenumerate}
\item \label{it:U11}
For any two recurring elements $\overline{\alpha},\overline{\beta}\in L_2^\N$ with $\overline{\alpha}<\overline{\beta}$ there exists a recurring element $\overline{\gamma}\in L_2^\N$ such that $\overline{\alpha}<\overline{\gamma}<\overline{\beta}$.
\item \label{it:U12}
$L_2^\N$ contains a subsemilattice isomorphic to $\Q$ consisting of recurring elements.
\item \label{it:U13}
$L_2^\N$ contains uncountably many non-isomorphic countable subdirect powers, each of which
 consists of recurring elements, is a linearly ordered set, and contains $\overline{0}$ and $\overline{1}$.
\item \label{it:U14}
The \sdpt\ of $L_2$ is $\fc$.
\end{thmenumerate}
\end{prop}

\begin{proof}
\ref{it:U11}
Observe that for any finite sequence $\sigma$, we have $\overline{\sigma}=\overline{\tau}$, where $\tau$ is any finite concatenation of copies of $\sigma$. Therefore, we may assume without loss that $\alpha$ and $\beta$ have the same length, say $\alpha=(a_1,\dots,a_k)$, $\beta=(b_1,\dots,b_k)$.
The assumption $\overline{\alpha}<\overline{\beta}$ means that $a_i\leq b_i$ for all $i=1,\dots,k$, and that at least one inequality is strict, say $a_j=0$, $b_j=1$.
Now for
\[
\gamma:=(a_1,\dots, a_{j-1},0,a_{j+1},\dots,a_k,a_1,\dots,a_{j-1},1,a_{j+1},\dots,a_k),
\]
we have $\overline{\alpha}<\overline{\gamma}<\overline{\beta}$.

\ref{it:U12}
 This follows by a standard argument: start with $\overline{0}$, $\overline{1}$, and iteratively insert a recurring element (whose existence is guaranteed by \ref{it:U11}) between any two points previously constructed. In the limit, this will create a countable, dense linear order with two end-points 
 $\overline{0}$, $\overline{1}$.
 Removing the endpoints leaves us with a copy of $\Q$.
 
 \ref{it:U13}
 The copy of $\Q$ identified in \ref{it:U12} contains copies of all countable linear orders, of which there are uncountably many up to isomorphism.
 Adding back $\overline{0}$, $\overline{1}$ to each of them yields the required family of subdirect powers.
 
 \ref{it:U14}
 This is an immediate consequence of \ref{it:U13}.
\end{proof}

We now turn our attention to an arbitrary finite, non-trivial semilattice $S$.
The idea here is to use the uncountably many subdirect powers guaranteed by Proposition \ref{prop:U1}, and 
`insert' each of them in a suitable part of the diagonal copy $\Delta_S$ of $S$, to yield uncountably many subdirect powers of $S$ in $S^\N$, again with additional desired properties to facilitate the proof of 
Proposition \ref{prop:comm2ids} in the next section.
We now describe the actual construction embodying the above idea.

\begin{cons}
\label{con:ins}
Let $S$ be a finite, non-trivial semilattice.
Let $0$ denote the smallest element of $S$, which must exist because $S$ is finite.
Also, let $e$ denote an arbitrary minimal non-zero element of $S$; this must exist because $S$ is finite and non-trivial.
The set $\{0,e\}$ is a two element subsemilattice of $S^\N$; with a slight abuse of notation let us denote it by $L_2$.
Now, for a subsemilattice $P$ of $L_2^\N$ which contains $\overline{0}$ and $\overline{e}$ let us define 
$\widetilde{P}:=P\cup\Delta_S$.
Intuitively, $\widetilde{P}$ is obtained by taking the semilattice $S$, and replacing the two element subsemilattice $L_2=\{0,e\}$ by an interval isomorphic to $P$; see Figure \ref{fig:tilde}.
\end{cons}

\begin{lemma}
\label{la:tildesub}
$\widetilde{P}$ is a subdirect power of $S$  in $S^\N$.
\end{lemma}

\begin{proof}
Each of $P$ and $\Delta_S$ is a subsemigroup of $S^\N$.
Now consider arbitrary $\pi\in P$ and $\sigma=\overline{s} \in \Delta_{S}$.
If $s\geq e$ then $se=es=e$, and hence $\sigma\pi=\pi\sigma=\pi\in \widetilde{P}$.
Similarly,  if $s\ngeq e$, then $se=es=0$ because of minimality of $e$, and hence $\sigma\pi=\pi\sigma=\overline{0}\in\widetilde{P}$.
It follows that $\widetilde{P}$ is a subsemigroup, and the elements from $\Delta_S$ ensure it is a subdirect power.
\end{proof}

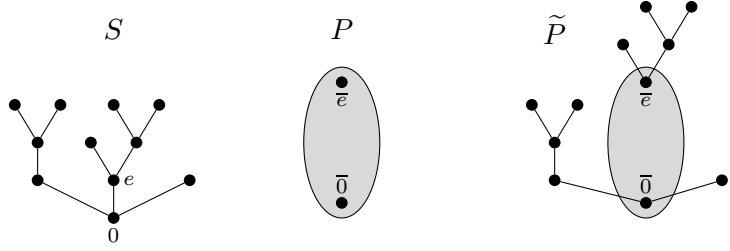
\begin{figure}

\begin{tikzpicture}[vertex/.style={circle,fill,draw, inner sep=0pt,minimum size=4pt}]

\node at (0,2.5) {$S$};

\node [vertex] (0) at (0, 0) {};
\node [vertex] (e) at (0, 0.5) {};
\node [vertex] (v1) at (-0.3, 1) {};
\node [vertex] (v2) at (0.3, 1) {};
\node [vertex] (v3) at (0, 1.5) {};
\node [vertex] (v4) at (0.6, 1.5) {};

\node [vertex] (u1) at (-1, 0.5) {};
\node [vertex] (u2) at (-1, 1) {};
\node [vertex] (u3) at (-1.3, 1.5) {};
\node [vertex] (u4) at (-0.7, 1.5) {};

\node [vertex] (w1) at (1, 0.5) {};

\node at (0,0) [below] {\scriptsize $0$};
\node at (0,0.5) [right] {\scriptsize $e$};

 \path (0) edge (e);
 \path (e) edge (v1);
 \path (e) edge (v2);
\path (v2) edge (v3);
\path (v2) edge (v4);

\path (0) edge (u1);
\path (u1) edge (u2);
\path (u2) edge (u3);
\path (u2) edge (u4);

\path (0) edge (w1);

%%%%%%%%

\node at (3,2.5) {$P$};

\node [ellipse,draw,fill=gray!30,minimum width = 10mm, minimum height=20mm] at (3,1) {};
\node [vertex] at (3, 0.2) {}; \node at (3,0.2) [above] {\scriptsize $\overline{0}$};
\node [vertex] at (3, 1.8) {}; \node at (3,1.8) [below] {\scriptsize $\overline{e}$};

%%%%%%%%

\node at (5.8,2.5) {$\widetilde{P}$};

\node [ellipse,draw,fill=gray!30,minimum width = 10mm, minimum height=20mm] at (7,1) {};
\node [vertex] (t0) at (7, 0.2) {}; \node at (7,0.2) [above] {\scriptsize $\overline{0}$};
\node [vertex] (te) at (7, 1.8) {}; \node at (7,1.8) [below] {\scriptsize $\overline{e}$};

\node [vertex] (tv1) at (6.7, 2.3) {};
\node [vertex] (tv2) at (7.3, 2.3) {};
\node [vertex] (tv3) at (7, 2.8) {};
\node [vertex] (tv4) at (7.6, 2.8) {};

\node [vertex] (tu1) at (5.8, 0.5) {};
\node [vertex] (tu2) at (5.8, 1) {};
\node [vertex] (tu3) at (5.5, 1.5) {};
\node [vertex] (tu4) at (6.1, 1.5) {};

\node [vertex] (tw1) at (8, 0.5) {};

 \path (te) edge (tv1);
 \path (te) edge (tv2);
\path (tv2) edge (tv3);
\path (tv2) edge (tv4);

\path (t0) edge (tu1);
\path (tu1) edge (tu2);
\path (tu2) edge (tu3);
\path (tu2) edge (tu4);

\path (t0) edge (tw1);

\end{tikzpicture}

\caption{The construction of $\widetilde{P}$ from $S$ and $P$.}
\label{fig:tilde}
\end{figure}

\begin{lemma}
\label{lemma:nonisotilde}
If $P_1$ and $P_2$ are non-isomorphic infinite subsemilattices of $L_2^{\N}$
containing $\overline{0}$ and $\overline{e}$, then $\widetilde{P}_1\not\cong\widetilde{P}_2$.
\end{lemma}

\begin{proof}
For an element $\sigma\in \widetilde{P}_i$ ($i=1,2$), let
$\sigma^\downarrow:=\{\tau\in\widetilde{P}_i\::\: \tau\leq\sigma\}$, the principal ideal generated by $\sigma$, and
$\sigma^\uparrow:=\{\tau\in\widetilde{P}_i\::\: \tau\geq\sigma\}$, the principal filter generated by $\sigma$.
In what follows, it will be clear from context whether these sets are meant to be taken in $\widetilde{P}_1$ or $\widetilde{P}_2$.
Note that
\begin{equation}
\label{eq:nit1}
\overline{e}^\downarrow=P_i,\ \overline{e}^\uparrow=\{ \overline{s}\::\: s\in S,\ s\geq e\}.
\end{equation}

Now suppose that there is an isomorphism $f: \widetilde{P}_1\rightarrow\widetilde{P}_2$.
From \eqref{eq:nit1} it follows that, in each $\widetilde{P}_i$, the element $\overline{e}$ is the unique element whose principal ideal is infinite, and whose principal filter has size $|\{s\in S\::\: s\geq e\}|$.
Therefore we must have $f(\overline{e})=\overline{e}$.
But then
\[
f(P_1)=f(\overline{e}^\downarrow)=f(\overline{e})^\downarrow=\overline{e}^\downarrow=P_2,
\]
contradicting the assumption $P_1\not\cong P_2$, and proving the lemma.
\end{proof}

\begin{prop}
\label{prop:slatt}
Let $S$ be a non-trivial finite semilattice.
Then $S^\N$ contains continuum many non-isomorphic subdirect powers of $S$, each of which consists of recurring elements and contains the diagonal copy $\Delta_{S}$ of $S$.
In particular, $S$ has \sdpt\ $\fc$.
\end{prop}

\begin{proof}
The result follows by taking \NRrev{$\fc$} many subdirect powers of $L_2$ guaranteed by Proposition \ref{prop:U1},
applying the above construction to each of them, and appealing to Lemmas 
\ref{la:tildesub} and \ref{lemma:nonisotilde}.
\end{proof}

\begin{rem}
Had our sole aim been to prove that non-trivial semilattices have \sdpt\ $\fc$, we could have accomplished it much more
rapidly then the preceding argument by using Birkhoff's Subdirect Representation Theorem; e.g.\ see \cite[Theorem 4.44]{ALV87}. 
According to this theorem, every semilattice is a subdirect product of subdirectly irreducible semilattices. As is well known, the two-element semilattice
$L_2$ is in fact the only subdirectly irreducible semilattice (e.g.\ see \cite[Corollary 2 (i), p.\ 172]{ALV87}), and this proves  Proposition \ref{prop:U1} \ref{it:U14}.
One can then deploy a variety of constructions, one of them being our Construction \ref{con:ins}, to extend this to every finite semilattice having \sdpt\ $\fc$.
However, we do need the extra information contained in Proposition \ref{prop:slatt} to facilitate the argument in the next section.
\end{rem}

\section{Commutative semigroups with more than one idempotent}
\label{sec:sevids}

In this section we prove that the \sdpt\ of every finite commutative semigroup with at least two idempotents is $\fc$.

Let $S$ be such a semigroup, and let $E=E(S)$ be its semilattice of idempotents.
Since $S$ is finite, each archimedean component of $S$ contains precisely one idempotent from $E$.
Therefore we can index the archimedean components of $S$ as $A_e$, for $e\in E$, where $e\in A_e$.

Since $|E|>1$, by Proposition \ref{prop:slatt} the semilattice $E^\N$ contains uncountably many non-isomorphic subdirect powers of $E$,
each of which contains the diagonal copy of $E$, and consists of recurring elements.
To prove that $S$ has \sdpt\ $\fc$ we will associate to each such subdirect power of $E^\N$ a subdirect power of $S^\N$, and prove that they are all non-isomorphic.

Let $\phi : S\rightarrow E$ be the natural epimorphism, specifically $\phi(s)=e$ when $s\in A_e$.
Extend $\phi$ in a component-wise manner to epimorphisms $S^k\rightarrow E^k$ ($k\in\N$) and $S^\N\rightarrow E^\N$;
these mappings will also be denoted by $\phi$ by a slight abuse of notation.
Now, for a subset $U\subseteq E^\N$ consisting of recurring elements, let
\[
\widehat{U}= \{ \overline{\sigma}\::\: \sigma\in S^k,\ k\in\N,\ \overline{\phi(\sigma)}\in U\}.
\]
Intuitively, we can think of $\widehat{U}$ as follows:
take in turn every $\tau=(e_1,\dots, e_k)\in E^k$ ($k=1,2,\dots$) such that $\overline{\tau}\in U$, 
and replace each $e_i$ by all possible elements of its archimedean component $A_{e_i}$.
For each $\sigma\in S^k$ thus obtained, $\overline{\sigma}$ is an element of $\widehat{U}$.

\begin{lemma}
\label{la:hatsub}
If $U$ is a subsemilattice of $E^\N$ and consists of recurring elements, then $\widehat{U}$ is a subsemigroup of $S^\N$,
with  $U$ its  semilattice of idempotents.
\end{lemma}

\begin{proof}
Let $\overline{\sigma},\overline{\tau}\in \widehat{U}$, with $\overline{\phi(\sigma)},\overline{\phi(\tau)}\in U$.
By replacing $\sigma$ and $\tau$ by appropriate concatenations, we can assume without loss that they have equal lengths. But then
\[
\overline{\phi(\sigma\tau)}=\overline{\phi(\sigma)\phi(\tau)}=\overline{\phi(\sigma)}\,\overline{\phi(\tau)}\in U,
\]
because $U$ is a subsemilattice. It follows that $\overline{\sigma}\,\overline{\tau}=\overline{\sigma\tau}\in \widehat{U}$,
and so $\widehat{U}$ is a subsemigroup. An element $\overline{\sigma}\in \widehat{U}$ is an idempotent if and only if all its entries are idempotents, which is equivalent to $\overline{\sigma}=\phi(\overline{\sigma})=\overline{\phi(\sigma)}\in U$.
\end{proof}

\begin{prop}
\label{prop:comm2ids}
Every finite commutative semigroup with at least two idempotents has \sdpt\ $\fc$.
\end{prop}

\begin{proof}
Let $S$ be such a semigroup, and retain all the notation from above.
By Proposition \ref{prop:slatt}, there exists a family $U_i$ ($i\in I$) of $\fc$ many pairwise non-isomorphic subdirect powers of $E$ in $E^\N$
consisting of recurring elements.
By definition, each $\widehat{U}_i$ also consists of recurring elements, and hence is countable.
Also, each $\widehat{U}_i$ is a subdirect power of $S$, because in our construction of $\widehat{U}_i$ each occurrence of some $e\in E$ in some tuple in $U_i$ is at some point replaced by all elements from its archimedean component, and the archimedean components partition $S$. Finally, the semigroups $\widehat{U}_i$ ($i\in I$) are pairwise non-isomorphic by Lemma \ref{la:hatsub}, because $U_i$ can be recovered from $\widehat{U}_i$ as its semilattice of idempotents. 
\end{proof}

\section{Nilpotent semigroups}
\label{sec:nil>2}

We now turn our attention to nilpotent semigroups. The status of those of nilpotency class $2$, i.e.\ of null semigroups, is easy to resolve:

\begin{prop}
\label{prop:null}
Every finite null semigroup has \sdpt \, $\aleph_{0}$.
\end{prop}

\begin{proof}
A null semigroup is determined up to isomorphism by its size, and so there are only countably many countable null semigroups. 
\end{proof}

For the rest of this section we will be concerned with nilpotent semigroups $S$ of class greater than $2$.
It turns out that we do not need to assume commutativity in our considerations.
For two elements $s,t\in S$ we say that $t$ is a (left) divisor of $s$ if $s=tu$ for some $u\in S$.
Denote by $\Div_S(s)$ the set of all divisors of $s$.

%Let $S$ be a finite nilpotent semigroup of class $k$.
%The sets
%\[
%I_p(S):= \underbrace{SS\dots S}_{p}\quad (p=1,2,\dots)
%\]
%are ideals of $S$ and $I_p(S)=\{0\}$ for $p\geq k$.
%Note that $I_p(S)I_q(S)\subseteq I_{p+q}(S)$ for all $p,q\in \N$.
%The \emph{layers} of $S$ are defined as:
%\[
%L_p(S):=I_p(S)\setminus I_{p+1}(S)\ (p=1,\dots,k-1),\quad L_p(S)=\{ 0\} \ (p\geq k).
%\]
%They are all non-empty, and the first $k$ partition $S$. 
%

%The elements of $L_1(S)$ have no divisors, while all the others do.
%In fact, an element $s\in L_p(S)$ must have a divisor in every layer $L_i(S)$ for $i=1,\dots,p-1$.
%%
%The direct power $S^\N$ is an uncountable nilpotent semigroup, also of class $k$. Its ideals 
%are $I_p(S^\N)=I_p(S)^\N$ ($p\in\N$), while the layers are given by
%\[
%L_p(S^\N)=
%\bigl\{ (s_1,s_2,\dots)\in S^\N\::\: (\forall i)(s_i\in I_p(S))\ \&\ (\exists i)(s_i\in L_p(S))\bigr\}\quad (p\in\N).
%\]

Now assume that $S$ is finite, and let its nilpotency class be $k>2$.
We aim to show that the \sdpt\ of $S$ is $\fc$.
To this end we introduce a construction which will yield a family of continuum many non-isomorphic countable subdirect powers of $S$.

\begin{cons}
Let $S$ be a finite nilpotent semigroup of nilpotency class $k>2$.
Thus there exists
$x=s_1\dots s_{k-1}\neq 0$, with $s_1,\dots,s_{k-1}\in S$, and we must have $xS=Sx=\{0\}$.
Keeping in mind that $k>2$, let $y:= s_1\dots s_{k-2}$, so that $y\in\Div_S(x)$.

Now consider $S^\N$. It is an uncountable nilpotent semigroup of class $k$.
Define
\begin{align*}
\sigma(i,s)&:=(\underbrace{0,\dots,0}_{i-1},s,0,0,\dots) \quad (i\in\N,\ s\in S),\\
\chi(i,j) &:= (\underbrace{0,\dots,0}_{i-1},y,\underbrace{x,\dots,x}_{j},0,0,\dots)\quad (i,j\in\N).
\end{align*}
For an infinite subset $M=\{ m_1,m_2,\dots\}\subseteq \N$, where $m_1<m_2<\dots$, let
\[
T_M :=  \{ \sigma(i,s)\::\: i\in\N,\ s\in S\}\cup
 \{ \chi(i,j) \: :\: i\in\N,\ 1\leq j\leq m_i\}.
\]
\end{cons}

The only non-zero products of elements from $T_M$ are:
\begin{align}
\label{eq:mulSM1}
\sigma(i,s)\sigma(i,t)&=\sigma(i,st) \quad (\text{if } st\neq 0),\\
\label{eq:mulSM2}
\sigma(i,s)\chi(i,j) &= \sigma(i,sy)\quad (\text{if } sy\neq 0),\\
\label{eq:mulSM3}
\chi(i,j)\sigma(i,s)&=\sigma(i,ys)\quad (\text{if } ys\neq 0),\\
\label{eq:mulSM4}
\chi(i,j)\chi(i,l)&=\sigma(i,y^2)\quad (\text{if } y^2\neq 0).
\end{align}
We then immediately conclude:

\begin{lemma}
\label{la:SMsub1}
$T_M$ is a subdirect power of $S$ in $S^\N$.
\qed
\end{lemma}

\begin{lemma}
\label{la:nodiv}
For $s\in S\setminus\{0\}$, $i\in\N$, $1\leq j \leq m_i$, we have
\begin{align*}
\bigl|\Div_{T_M}(\sigma(i,s))\bigr|& = \begin{cases} |\Div(s)| &\textup{  if } y\not\in\Div_S(s),\\
                                                             |\Div(s)|+m_i &\textup{  if } y\in\Div_S(s),
                                       \end{cases}\\
\bigl|\Div_{T_M}(\chi(i,j))\bigr|    &= 0.                          
\end{align*}
\end{lemma}

\begin{proof}
From \eqref{eq:mulSM1} we see that $\sigma(i,s)$ certainly has divisors $\sigma(i,t)$, where $t\in \Div_S(s)$.
Furthermore, by \eqref{eq:mulSM3}, it will also have all $\chi(i,j)$ ($j=1,\dots,m_i$) as divisors, provided $y\in \Div_S(s)$.
Finally no $\chi(i,j)$ is a product of two elements of $T_M$, i.e.\ it has no divisors.
\end{proof}

\begin{prop}
\label{prop:nilp}
A finite nilpotent semigroup of class greater than $2$ has \sdpt\ $\fc$.
\end{prop}

\begin{proof}
Let $S$ be such a semigroup, with all the notation as in the preceding discussion.
Let $n:=|S|$.
The semigroups $ T_M$, where $M$ is an infinite subset of $\{ n,2n,3n,\dots\}$, provide 
a family of continuum many countable subdirect powers by Lemma~\ref{la:SMsub1}.
We prove that they are pairwise non-isomorphic.
To this end, let $M$ and $P$ be two distinct subsets of $\{n,2n,\dots\}$.
Let the elements of $M$ be $m_1<m_2<\dots$, and let those of $P$ be $p_1<p_2<\dots$.
Without loss assume that $m_i\not \in P$ for some $i$.
By Lemma \ref{la:nodiv}, the element $\sigma(i,x)$ has $|\Div_S(x)|+m_i$ divisors, because $y\in \Div_S(x)$ by choice of $x$ and $y$.
We claim that no non-zero element of $T_P$ has this number of divisors.
By Lemma \ref{la:nodiv}, possible numbers of divisors for such an element are $0$, $|\Div_S(s)|$ and $|\Div_S(s)|+p_j$,
where $s\in S\setminus\{0\}$, $j\in \N$.
Since $n=|S|$, we have
\[
|\Div_S(x)|+m_i>n>|\Div_S(s)|\geq 0.
\]
Also, remembering that $m_i,p_j\in\{n,2n,\dots\}$ and $m_i\neq p_j$, we have
\[
|m_i - p_j| \geq n > |\Div_S(s)|-|\Div_S(x)|,
\]
which implies
\[
|\Div_S(x)|+m_i\neq |\Div_S(s)|+p_j.
\]
Therefore $T_M$ and $T_P$ cannot be isomorphic, and the proposition is proved.
\end{proof}

\section{Ideal extensions of groups by nilpotent semigroups}
\label{sec:nilext}

In this section we prove that every finite ideal extension of a non-trivial group by a non-trivial nilpotent semigroup has \sdpt \, $\fc$. Given such a semigroup, we again will detail a construction similar to the nilpotent case in order to do this. Yet again, there is no need to assume commutativity.

\begin{cons} Let $S$ be a finite semigroup such that its minimal ideal is a non-trivial group $G$, and the quotient $S/G$ is nilpotent of class $k>1$. Thus we have
\[
s_1s_2\dots s_k\in G \quad \text{for all } s_1,\dots,s_k\in S,
\]
but  there exist some $s_1',\cdots ,s_{k-1}'\in S$ such that
\[
x:=s_1'\dots s_{k-1}'\not\in G.
\]
Notice that $xS\cup Sx\subseteq G$. In particular, $x^t\in G$ for all $t\geq 2$.
%Denote by $e$ the identity of $G$, which is also the unique idempotent of $S$.
%The nilpotency assumption for $S/G$ means that 
%\[
%s_1s_2\dots s_k\in G \quad \text{for all } s_1,\dots,s_k\in S.
%\]
%
%Define the \emph{layers} of $S$ in the same way as they were defined in the nilpotent case in the previous section, the only difference being that now we stipulate
%\[
%L_k(S)=L_{k+1}(S)=\dots:=G.
%\]
Denote by $e$ the identity element of $G$, which is also the unique idempotent of $S$.
Let $\underline{x}:=ex\in G$, and fix an arbitrary $g\in G$ with $g\neq \underline{x}$.

The direct power $S^\N$ is again an ideal extension of a group by a $k$-nilpotent semigroup. 
Its minimal ideal is the group $G^\N$, but the quotient 
$S^\N/G^\N$ is not the same as $(S/G)^\N$.

We want to exhibit a family of continuum many countable subdirect powers of $S$ in $S^\N$.
To this end, for any
infinite subset $M=\{ m_1,m_2,\dots\} \subseteq \N$,
with  $m_1<m_2<\dots$, define
\[
W_M:= G^\infty\cup \Delta_S\cup U_M,
\]
where
\begin{align*}
G^\infty &:= \bigl\{ (g_1,g_2,\dots)\in G^\N \::\: g_{p+1}=g_{p+2}=\dots \text{ for some } p\geq 0\bigr\},\\
\Delta_S &:= \{ (s,s,\dots)\::\: s\in S\},\\
U_M&:=\bigl\{ (\underbrace{e,\dots,e}_{p-1},g,\underbrace{x,\dots,x}_{q},\underline{x},\underline{x},\dots)\::\: 
p\in\N,\ 1\leq q\leq m_p\bigr\}.
\end{align*}
\end{cons}

\begin{lemma}
\label{la:SMsub}
With the notation as above, each $W_M$ is a countable subdirect power in~$S^\N$.
\end{lemma}

\begin{proof}
That $W_M$ is a subsemigroup of, and subdirect power in, $S^\N$ follows from the following facts:
\begin{nitemize}
\item
$G^\infty\leq G^\N$;
\item
$\Delta_S$ is a subdirect power in $S^\N$;
\item
$G^\infty\cup  \Delta_S\leq S^\N$;
\item
$W_M U_M\cup U_MW_M\subseteq G^\infty$.
\end{nitemize} 
Countability is obvious.
\end{proof}

To demonstrate that certain $W_M$ are pairwise non-isomorphic, we will use \emph{roots}, defined as follows.
Let $m:=|G|$. Recall that 
$h^m=e$ and $h^{m+1}=h$ for all $h\in G$,
and hence
$\sigma^m=\overline{e}$ and $\sigma^{m+1}=\sigma$ for all $\sigma\in G^\N$.
For $\sigma\in G^\infty$ let
\[
\sqrt[m+1]{\sigma}:=\bigl\{ \tau\in W_M\setminus G^\infty\::\: \tau^{m+1}=\sigma\bigr\}.
\]
%We analogously define $\sqrt[m+1]{s}$ for $s\in S$.

\begin{lemma}
\label{la:countroots}
Suppose $|S|=n$, and let $M$ be an infinite subset of $\{n+1,n+2,\dots\}$. With the notation as above, we have
\[
\bigl\{ \bigl|\sqrt[m+1]{\sigma}\bigr|\::\: \sigma\in W_M \bigr\}\cap \bigl\{ n+1,n+2,\dots\bigr\} =M.
\]
\end{lemma}

\begin{proof}
The lemma is proved by computing the $(m+1)$st powers of all elements of 
$W_M\setminus G^\infty= (\Delta_S\setminus\Delta_G)
\cup U_M$. 

First, however, consider the element $x^{m+1}$. Since $m\geq 2$, we have $x^{m+1}\in G$, and so
\begin{equation}
\label{eq:xm1}
x^{m+1}=ex^{m+1}=exex^m=\dots =(ex)^{m+1}=\underline{x}^{m+1}.
\end{equation}

Now, for $\sigma\in \Delta_S\setminus\Delta_G$ we have $\sigma^{m+1}\in \Delta_S$,
whereas for
\begin{equation}
\label{eq:sigma}
\sigma=(\underbrace{e,\dots,e}_{p-1},g,\underbrace{x,\dots,x}_{q},\underline{x},\underline{x},\dots)\in U_M
\end{equation}
we have, using \eqref{eq:xm1},
\begin{equation}
\label{eq:roots}
\sigma^{m+1}=(\underbrace{e,\dots,e}_{p-1},g,\underline{x},\underline{x},\dots).
\end{equation}
Note that this last element is never in $\Delta_{S}$, because $g\neq \underline{x}$.
From this, it first of all follows that the elements in $\Delta_S\setminus\Delta_G$ have $\leq n$ roots.
And then
it follows that the elements having more than $n$ roots are precisely those appearing in \eqref{eq:roots}.
The number of roots this element has is precisely the number of possible values of $q$ in \eqref{eq:sigma},
and this is equal to~$m_p$.
\end{proof}

\begin{prop}
\label{prop:ext}
Every finite semigroup which is an ideal extension of a non-trivial group by a non-trivial nilpotent semigroup has \sdpt\ $\fc$.
\end{prop}

\begin{proof}

For such a semigroup $S$ of order $n$, letting $M$ range over all infinite subsets of $\{n+1,n+2,\dots\}$
yields continuum many semigroups $W_M$. They are all subdirect powers of $S$ by Lemma \ref{la:SMsub},
and are pairwise non-isomorphic by Lemma \ref{la:countroots}.

\end{proof}

%\begin{proof}
%For such a semigroup $S$ (with maximal subgroup $G$) and any two infinite subsets $M,P \subseteq \{n+1,n+2,\dots\}$ with $n = |S|$, $m = |G|$,  and $M \not = P$, without loss we can choose some $k \in M\setminus P$. 
%Then for $\E(S,M)$ and $\E(S,N)$, which are countable subdirect powers in $S^{\N}$ by Lemma \eqref{la:SMsub}, there is some element of $\E(S,M)$ with $k$ $(m+1)$-th roots, but no such element of $\E(S,P)$ by \eqref{la:countroots}. 
%This means the two subdirect products are non-isomorphic, and consequently we have uncountably many such which are pairwise non-isomorphic when ranging over all infinite subsets of $\{n+1,n+2,\dots\}$.
%\end{proof}
%

\section{Concluding remarks}
\label{sec:conc}

We have seen that unlike in groups, where having \sdpt\ $\aleph_{0}$ is equivalent to being commutative by the results of Hickin and Plotkin \cite {hickin1981} and McKenzie \cite{mckenzie1982} (Theorem \ref{thm:hpm}), for semigroups this is not the case: `most' commutative semigroups in fact have uncountable \sdpt\ according to our Main Theorem. A natural question arises to classify all semigroups with countable \sdpt. In particular, one might wonder whether such semigroups must be commutative. This, however, is easily seen not to be the case: semigroups of left or right zeros (i.e. semigroups satisfying $xy=x$ or $xy=y$ for all $x,y$) are non-commutative, but are easily seen to have countable \sdpt, by an argument analogous to the proof for null semigroups (Proposition \ref{prop:null}).
More generally, the same holds for \emph{rectangular bands}, i.e. semigroups of the form $S=I\times J$, where $I$ and $J$ are any sets, and multiplication is given by $(i,j)(k,l)=(i,l)$. Clearly, the isomorphism type of a rectangular band is completely determined by the cardinalities $|I|$ and $|J|$.
An alternative description of rectangular bands is that they are precisely the semigroups satisfying $x^2=x$ and $xyz=xz$ for all $x,y,z$ \cite[Theorem 1.1.3]{howie95}.
It follows that subsemigroups and direct products of rectangular bands are again rectangular bands. Combining these observations together we obtain: 

\begin{prop}
\label{prop:rect}
Every finite rectangular band has \sdpt\ $\aleph_{0}$.
\qed
\end{prop}

Note that a rectangular band $S=I\times J$ is commutative if and only if $|I|=|J|=1$.

One might also wonder whether the direct product of two finite semigroups of \sdpt\ $\aleph_{0}$ also necessarily has \sdpt\ $\aleph_{0}$.
However, again, this is not the case: the direct product of a non-trivial abelian group $G$ by a non-trivial null semigroup is a commutative semigroup and is an extension of $G$ by a (larger) null semigroup, and hence has \sdpt\ $\fc$ by the Main Theorem.

As we mentioned in the Introduction, both Hickin and Plotkin \cite{hickin1981} and McKenzie \cite{mckenzie1982} prove results that are more general than Theorem \ref{thm:hpm}. In particular, McKenzie's result states that if $G$ is a non-abelian group of any cardinality, 
and if $\kappa\geq |G|$ is an infinite cardinal, then $G$ has $2^\kappa$ non-isomorphic subdirect powers of cardinality $\kappa$. We may ask:

\begin{que}
Is it true that for every semigroup $S$, the number of subdirect powers of $S$ of cardinality $\kappa\geq |S|$ is either $\kappa$ or $2^\kappa$?
\end{que}

The above is not the case for \emph{unary algebras}, as was shown in \cite[Proposition 5.3]{ruwi}.
Specifically, for the monounary algebra $(\{0,1,2\},f)$ where $f(x)=\max(x-1,0)$, the number of non-isomorphic subdirect powers of cardinality $\kappa\geq\aleph_0$ is equal to the number of cardinals $\beta\leq\kappa$.
Unary algebras can of course be viewed as monoids acting on sets, and the main result of \cite{ruwi} can be interpreted as a classification of finite monoid acts of \sdpt\ $\aleph_{0}$. No further classifications are known to the best of our knowledge, and situation for rings, associative and Lie algebras, loops and lattices seems to be particularly worth investigating.

The question of the number of countable subdirect powers for an algebra is related to the notion of boolean separation.
We say that an algebraic structure $A$ is \emph{boolean separating} if $A^{B_1}\cong A^{B_2}$ implies $B_1\cong B_2$ for any boolean algebras $B_1$, $B_2$. Here $A^B$ denotes the \emph{boolean power} of the algebra $A$ by a boolean algebra $B$;
see \cite[Section IV.5]{burris}.
Since boolean powers are subdirect powers and since there are uncountably many countable boolean algebras, it follows that every boolean separating finite algebra has \sdpt\ $\aleph_{0}$. Finite boolean separating groups $G$ have been classified in \cite{apps82}.
In particular, when $G$ is subdirectly irreducible, it is boolean separating if and only if it is non-abelian \cite{lawrence81}.
Motivated by this, and by our work in this paper, we pose:

\begin{prob}
Classify finite boolean separating (commutative) semigroups.
\end{prob}

%\section{Things to consider}
%
%\begin{nitemize}
%\item
%The results at the moment are framed as countable vs. uncountable. We should probably go for countable vs. continuum many.
%\item
%Is it true that for a finite commutative semigroup $S$ and a cardinal $\kappa\geq \aleph_0$, the number of subdirect powers of $S$ of cardinality $\kappa$ is either $\kappa$ or $2^\kappa$?
%\item
%Does there exist a countable commutative semigroup, which is neither a group nor a zero semigroup, and which has \ACrev{\sdpt\ $\aleph_{0}$}?
%\NRcomm{
%I don't think that taking countable subdirect powers of our semigroups of countable \sdpt\ will take us anywhere useful: they are still all abelian groups or null semigroups. I wonder whether some analysis via archimedean components may be more promising? E.g., if each archimedean component contains an idempotent, then perhaps the arguments from this paper carry over? And if there is a component without an idempotent then we have a copy of $\N$, so our $\N\times\N$ paper gives us uncountably many non-isomorphic subsemigroups. How to turn them into subdirect powers?}
%\end{nitemize}
%

\bibliographystyle{plain}

\end{document}